\documentclass[12pt,a4paper]{amsart}
\usepackage{amsfonts,amsmath,amsthm,amssymb}
\usepackage{pstricks,pst-node}
\usepackage{fullpage}
\usepackage{dsfont}

\newcommand\N{\mathds{N}}

\newtheorem{theorem}{Theorem}

\newtheorem{lemma}[theorem]{Lemma}
\newtheorem{problem}[theorem]{Problem}

\title{Short note on the convolution of binomial coefficients}

\author{Rui Duarte}
\address{Center for Research and Development in Mathematics and Applications,
Department of Mathematics,
University of Aveiro}
\email{rduarte@ua.pt}

\author{Ant\'onio Guedes de Oliveira} \address{CMUP and Mathematics
  Department, Faculty of Sciences, University of Porto}
\email{agoliv@fc.up.pt}

\DeclareMathOperator{\Bef}{Bef}
\DeclareMathOperator{\Aft}{Aft}

\date{\today}

\begin{document}

\begin{abstract}
  We know \cite{part} that, for every non-negative integer
  numbers $n,i,j$ and for every real number $\ell$,
\begin{align}
&\sum_{i+j=n}\binom{2\,i-\ell}{i}\binom{2\,j+\ell}{j}=\sum_{i+j=n}\binom{2\,i}{i}
\binom{2\,j}{j}\label{eq1},
\intertext{which is well-known to be $4^n$.  We extend this result
    by proving that, indeed,}
&\sum_{i+j=n}\binom{a\,i+k-\ell}{i}\binom{a\,j+\ell}{j}=
\sum_{i+j=n}\binom{a\,i+k}{i}\binom{a\,j}{j}\label{eq2}
\end{align}
for every integer $a$ and for every real $k$, and present new
expressions for this value.
\end{abstract}

\maketitle

\noindent
We consider the sequence $\big\{\binom{a\,n}{n}\big\}_{n=0}^\infty$, where $a$ is any integer number, negative, zero or positive, and take the convolution of this sequence with itself, defined by $P_a(n)=\sum_{i+j=n}\binom{a\,i}{i}\binom{a\,j}{j}$.

When $a=2$, the former is sequence A000984 of \cite{oeis}, the central binomial coefficients, and the latter is sequence $A000302$ of \cite{oeis}, the powers of $4$. In fact (cf. \cite{part}), this can be proved directly using \eqref{eq1}, and then the inclusion-exclusion principle. Note that
\begin{equation}
2\,P_2(n)=2^{2n+1}=\sum_{i=0}^{2n+1} \binom{2n+1}{i}=2\sum_{i=0}^n
\binom{2n+1}{i}.\label{eq3}
\end{equation}
For another identity, define as usual $[n]=\{1,\dotsc,n\}$ for any natural number $n$, and consider the collection of the subsets of $[2n]$ with more than $n$ elements with the same $(n+1)$-th element, say $p$. Note that $p=n+1+i$ for some $i=0,\dotsc,n-1$ and that there are $\binom{n+i}{n}\,2^{n-i-1}$ subsets in the collection. It follows that the number of all subsets of $[2n]$ is
\begin{equation}
P_2(n)=2^{2n}=2\sum_{i=0}^{n-1}2^{n-i-1}\binom{n+i}{i}+\binom{2n}{n}
=\sum_{i=0}^n 2^{n-i}\binom{n+i}{i}.\label{eq4}
\end{equation}

We generalize these identities, namely \eqref{eq1}, \eqref{eq3} and \eqref{eq4}. When $a=3$ and $a=4$, we have sequences $A006256$ and $A078995$ of \cite{oeis}, and no such simple formulas for $P_3(n)$ and $P_4(n)$ are known as in case $a=2$. For these sequences, we obtain, for every real $\ell$,
\begin{align*}
\sum_{i+j=n}\binom{3i}{i}\binom{3j}{j}
&=\sum_{i+j=n} 2^i\binom{3n+1}{j}=\sum_{i+j=n} 3^i\binom{2n+j}{j}
=\sum_{i+j=n}\binom{3i-\ell}{i}\binom{3j+\ell}{j}\\
\sum_{i+j=n}\binom{4i}{i}\binom{4j}{j}
&=\sum_{i+j=n} 3^i\binom{4n+1}{j}=\sum_{i+j=n} 4^i\binom{3n+j}{j}
=\sum_{i+j=n}\binom{4i-\ell}{i}\binom{4j+\ell}{j}
\end{align*}

More generally we obtain the following theorem.

\begin{theorem}\label{thm}
For every non-negative integer numbers $i$, $j$ and $n$, and for every
real numbers $k$ and $\ell$,
\begin{align}
\sum_{i+j=n}\binom{a\,i+k-\ell}{i}\binom{a\,j+\ell}{j}
&=\sum_{i+j=n}\binom{a\,i+k}{i}\binom{a\,j}{j} \nonumber \\
&=\sum_{i=0}^n (a-1)^{n-i}\,\binom{a\,n+k+1}{i} \label{eq6}\\
&=\sum_{i=0}^n a^{n-i}\,\binom{(a-1)n+k+i}{i} \label{eq7}
\end{align}
where we take $0^0=1$.
\end{theorem}

For the proof of this theorem we need some technical results.
\begin{lemma}\label{lemma1}
  Let, for any real $\ell$ and integers $a$ and $n$ such that $n\geq0$,
\begin{align*}
  &S_{a,\ell}(n)\;=\;\sum_{i=0}^n (-1)^i \binom{\ell-(a-1)i}{i}
  \binom{\ell-a\,i}{n-i}\\[-15pt]
  \intertext{Then}\\[-20pt]
  &\sum_{i=0}^n\binom{n}{p}\,S_{a,\ell}(p)\;=\;S_{a+1,\ell+n}(n).
\end{align*}
\end{lemma}
\begin{proof}
\begin{align*}
  \sum_{i=0}^n\binom{n}{p}\,S_{a,\ell}(p) &=\sum_{i=0}^n\left[(-1)^i
    \binom{\ell-(a-1)i}{i}
    \sum_{p=i}^n \binom{\ell-a\,i}{p-i}\binom{n}{p}\right]\\
  &=\sum_{i=0}^n \left[(-1)^i \binom{\ell-(a-1)i}{i} \sum_{p=i}^n
    \binom{\ell-a\,i}{\ell-(a-1)i-p}\binom{n}{p}\right]\\
  &=\sum_{i=0}^n (-1)^i \binom{\ell-(a-1)i}{i}
  \binom{\ell+n-a\,i}{\ell-(a-1)\,i}\\
  &=\sum_{i=0}^n (-1)^i \binom{(\ell+n)-a\,i}{i \ , \ n-i \ , \ \ell-ai} \\
  &=\sum_{i=0}^n (-1)^i \binom{(\ell+n)-a\,i}{i}
  \binom{(\ell+n)-(a+1)\,i}{n-\,i}
\end{align*}
where we use Vandermonde's convolution in the third equality.
\end{proof}

\begin{lemma}\label{lemma2}
With the notation of the previous lemma,
$$S_{a,\ell}(n)\;=\;(a-1)^n.$$
\end{lemma}
\begin{proof}
  First note that we may assume that $\ell$ is a natural number,
    since $S_{a,\ell}(n)$ is a polynomial in $\ell$, and thus is
    constant. Now, suppose that $S_{a,\ell}(p)=x^p$ for some numbers
  $a$, $\ell$, $p$ and $x$. Then, from Lemma~\ref{lemma1} it follows
  that $S_{a+1,\ell+n}(n)=(1+x)^n$. Hence, all we must prove is that
  $S_{a,\ell}(n)=0$ when $a=1$ and $\ell\in\N$.

  For this purpose, define $\mathcal{A}=\mathcal{A}_\varnothing$ as
  the set of $n$-subsets of the set $[\ell]=\{1,2,\dotsc,\ell\}$ and,
  for every non-empty subset $T$ of $[\ell]$,
  $\mathcal{A}_T=\big\{A\in\mathcal{A}\mid A\cap T=\varnothing\big\}$.
  Now, the result follows immediately from the inclusion-exclusion
  principle applied to this family.
\end{proof}

\begin{lemma}\label{lemma3}
Let $s$ and $t$ be positive integers. Then
$$\binom{s+t+1}{j} = \sum_{i=0}^j \binom{s-i}{s-j} \binom{t+i}{i}.$$
\end{lemma}

\begin{proof}
  Given a subset $S$ of $[n]$ with $k$ elements and $p \in [n]
  \setminus S$, let $\Bef_p (S) = S \cap [p-1]$ and $\Aft_p (S) = \{ t
  \in [n-p] \mid t+p \in S \}$.

  Now, let $A$ be a subset of $[s+t+1]$ with $j$ elements and
  $p(A)$ be the $s-j+1$ smallest element of $[s+t+1]$ which is not in
  $A$. In other words, $\# \{ x \in A \mid x < p(A) \} = j-i$ and $\#
  \{ x \in A \mid x > p(A) \} = i$. One can easily see that the
  mapping
$$
\begin{array}{cccc}
  \varphi: & \mathcal{P}_j ([s+t+1]) & \to & \bigcup \limits_{0 \leq i \leq j} \mathcal{P}_{j-i} ([s-i]) \times \mathcal{P}_i ([t+i]) \\
  & A & \mapsto & (\Bef_{p(A)} (A), \Aft_{p(A)} (A))\end{array}$$
is a bijection, with inverse given by $\psi(B,C)=B \cup \{c+\#C\mid
c\in C \}$, and the union is disjoint.
\end{proof}

\begin{proof}[Proof of Theorem~\ref{thm}]
Let $\mathfrak{S}=\sum_{i+j=n} \binom{a\,i+k-\ell}{i} \binom{a\,j+\ell}{j}
= \sum_{i+j=n} (-1)^i \binom{\ell-k'-(a-1)i}{i} \binom{a\,n+\ell-a\,i}{j}$,
with $k'=k+1$. Then, by Vandermonde's convolution,
\begin{eqnarray*}
  \mathfrak{S}
  & = & \sum_{i+j=n} \left[ (-1)^i \binom{\ell-k'-(a-1)i}{i} \sum_{p+m=j}
    \binom{a\,n+k'}{p} \binom{\ell-k'-a\,i}{m} \right] \\
  & = & \sum_{p=0}^n \left[ \binom{a\,n+k'}{p} \sum_{i+m=n-p} (-1)^i
    \binom{\ell-k'-(a-1)i}{i} \binom{\ell-k'-a\,i}{m} \right]
\end{eqnarray*}
Now, \eqref{eq6} follows immediately from Lemma~\ref{lemma2} and
\eqref{eq7} from Lemma~\ref{lemma3}.
\end{proof}

We end this article with a new result that, when we represent
by $\left(\!\!\binom{n}{k}\!\!\right)$ the number $\binom{n+k-1}{k}$
of $k$-multisets of elements of an $n$-set, can be formulated in the
following elegant terms.

\begin{theorem}\label{thm2}
For every real $\ell$ and integers $a,n,i,j$ such that $n,i,j\geq 0$,
$$\sum_{i+j=n} (-1)^i \left(\!\!\!\binom{\ell-a\,i}{i}\!\!\!\right)
\binom{\ell-a\,i}{j}= a(a-1)^{n-1}.$$
\end{theorem}
\begin{proof}
By Pascal's rule,
\begin{align*}
\sum_{i+j=n} (-1)^i \binom{\ell-1-(a-1)i}{i} \binom{\ell-a\,i}{j}
= & \sum_{i=0}^n (-1)^i \binom{\ell-(a-1)i}{i} \binom{\ell-a\,i}{n-i}\\
& - \sum_{i=1}^n (-1)^i \binom{\ell-(a-1)i-1}{i-1} \binom{\ell-a\,i}{n-i}\\
= & S_{a, \ell}(n)+S_{a, \ell-a}(n-1)
\end{align*}
\end{proof}

\begin{problem}
Give a full combinatorial proof of Theorem~\ref{thm2}.
\end{problem}

\noindent\textbf{Acknowledgments.}\enspace
The work of both authors was supported in part by the European
Regional Development Fund through the program COMPETE --Operational
Programme Factors of Competitiveness (``Programa Operacional Factores
de Competitividade'') - and by the Portuguese Government through FCT
-- Funda\c{c}\~ao para a Ci\^encia e a Tecnologia, under the projects
PEst-C/MAT/UI0144/2011 and PEst-C/MAT/UI4106/2011.

\end{document}